\documentclass{paper}
\usepackage[utf8]{inputenc}
\usepackage[T1]{fontenc} 

\usepackage{graphicx}
\usepackage{color}
\usepackage{amsmath}
\usepackage{amsthm}
\usepackage{amssymb}

\usepackage{dsfont}

\usepackage{xcolor}

\definecolor{boubelcolor}{rgb}{.65,0.05,0}

\DeclareUnicodeCharacter{2260}{\colorlet{memoireboubel}{.}\color{boubelcolor}} 
\DeclareUnicodeCharacter{00B1}{\color{memoireboubel}{}} 

\DeclareUnicodeCharacter{00A3}{\colorlet{memoirejuillet}{.}\color{blue}} 
\DeclareUnicodeCharacter{00A7}{\color{memoirejuillet}{}} 

\theoremstyle{plain}
\newtheorem{them}{Theorem}[section]
\newtheorem*{mthem}{Main Theorem}
\newtheorem{pro}[them]{Proposition}
\newtheorem{cor}[them]{Corollary}
\newtheorem{lem}[them]{Lemma}

\theoremstyle{definition}
\newtheorem{defi}[them]{Definition}

\theoremstyle{remark}
\newtheorem{rem}[them]{Remark}
\newtheorem{ex}[them]{Example}
\newtheorem*{notation}{Notation}

\newcommand{\X}{\mathcal{X}}

\newcommand{\R}{\mathbf{R}}
\newcommand{\N}{\mathbf{N}}

\newcommand{\id}{\mathrm{id}}
\newcommand{\ma}{\mathrm{Marg}}
\newcommand{\p}{\mathcal{P}}
\renewcommand{\d}{\mathrm{d}}

\newcommand{\spt}{\mathrm{Spt}}
\newcommand{\graph}{\mathrm{Graph}}
\newcommand{\card}{\mathrm{card}}
\newcommand{\pr}{\mathrm{proj}}

\newcommand{\eps}{\varepsilon}

\title{On a solution to the Monge transport problem on the real line arising from the strictly concave case}
\author{Nicolas Juillet}

\begin{document}
\maketitle
\begin{abstract}
It is well-known that the optimal transport problem on the real line for the classical distance cost may not have a unique solution. In this paper we recover uniqueness by considering the transport problems where the costs are a power smaller than one of the distance, and letting this parameter tend to one. A complete construction of this solution that we call excursion coupling is given. This is reminiscent to the one in the convex case. It is also characterized as the solution of secondary transport problems. Moreover, a combinatoric/geometric characterization of the routes used for this transport plan is provided.
\end{abstract}

We first introduce the mass transport problem and quickly arrive to our Main Theorem. The reason why it can appear so fast is that it concerns the most basic setting -- together with the finite discrete setting -- where the optimal transport problem can be introduced, namely the real line for a power cost.

\begin{defi}
Let $(\X,\d)$ be a metric space and $p$ be a positive real number. We denote by $\p_p(\X)$ the set of Borel probability measures $\mu\in \p(\X)$ such that $\int \d(x_0,x)^p\d\mu(x)<+\infty$ for some (and in fact any) $x_0\in \X$ and say that $\mu$ has finite moment of order $p$. The set $\ma(\mu,\nu)$ is the (convex) set of measures $\pi\in \p(\X\times \X)$ with first marginal $\mu$ and second marginal $\nu$, i.e $(\pr^1)_\#\pi=\mu$ and $(\pr^2)_\#\pi=\nu$ where $\pr^i$ is the $i$-th coordinate function. These definitions naturally extend to positive measures $\pi,\,\mu,\,\nu$ of finite positive mass $\mu(\X)$.

For $\mu$ and $\nu$ two measures with $\mu(\X)=\nu(\X)$ we call ``$L^p$ transport problem'' (of Monge and Kantorovich), the problem to minimize
\[T_p:\pi\in \ma(\mu,\nu)\mapsto \iint \d(x,y)^p\,\d\pi(x,y)\in \R\cup \{\infty\}.\]
We denote by $\ma_p^*(\mu,\nu)$ the (convex) set of solutions to this problem. 
\end{defi}

The space $\ma(\mu,\nu)$ is equipped with the weak convergence topology and it is compact according to Prokhorov Theorem, see \cite[\S1.1.7]{Vi1}. Moreover $T_p$ is continuous on $\ma(\mu,\nu)$
, so that $\ma_p^*(\mu,\nu)$ is not empty. One may have a look at page \pageref{lem:continuous} where this basic continuity result and a more evolved one are stated and proved. Note that if $\mu$ and $\nu$ have finite moment of order $p$ the value of the problem is finite, with $\min_{\pi\in \ma(\mu,\nu)}T_p(\pi)\leq T_p(\mu\otimes \nu)<+\infty$. Hence, we will assume that $\mu$ and $\nu$ have finite moment of order $1$ permitting a finite value for all the $L^p$ transport problems, for $p\in ]0, 1]$. Note moreover that for every $p\in ]0,1]$ since $T_p$ is continuous and $\ma(\mu,\nu)$ compact, the set $\ma_p^*(\mu,\nu)$ is not empty.


Let us review some of the known facts concerning the solutions of the $L^p$ transport problem on the real line $\X=\R$. See also Remark \ref{rem:geod} for the connexion of the $L^1$ transport problem on $\R$ with the $L^1$ transport problem on geodesic spaces.

\begin{itemize}
\item For $p>1$, except from degeneracy occurring when $\min_{\pi\in \ma(\mu,\nu)}T_p=\infty$ the set of solutions $\ma_p^*(\mu,\nu)$ is reduced to a single element, known as \emph{commonotonic coupling}, \emph{quantile coupling}, \emph{monotone rearrangement}, \emph{Hoeffding-Fr\'echet transport} or any combination of these vocables. The universality and usefulness of the notion is reflected by this number of names. 
\item The value $p=1$ is the one in the original problem of Monge of 1781 \cite{Monge} (except that Monge considers $\X=\R^2$ or $\R^3$). For this parameter the quantile transport is still one of the solutions but it may not be the unique one. There is actually a broad class of pairs $(\mu,\nu)$ such that the set of solutions $\ma_1^*(\mu,\nu)$ is the whole set of transport plans $\ma(\mu,\nu)$, see Remark \ref{rem:DiML}. Note also that the solutions of the Monge problem in many geodesic space $\X$ can be disintegrated along transport rays, i.e parts of $\X$ isometric to intervals where the obtained measures are solution of the Monge problem on the real line. This disintegration is the fundament for many powerful recent applications, see Remark \ref{rem:geod}
\item The problem for the values $p\in ]0,1[$ may be less known and understood as $p\geq 1$. However, Gangbo-McCann \cite{GMcC} and McCann \cite{McC} thoroughly explored this range of the parameter $p$. For a class of absolutely continuous measures $\mu$ and $\nu$, McCann found an algorithm to restrict the search for the solution -- it is unique under the absolute continuity assumption -- to a finite number of classes, where regions of $\R$ concerning $\mu$ are mapped onto other regions of $\R$ concerning $\nu$, the frontiers between the regions having to be determined. Note that for general values of $(\mu,\nu)$, the solution space $\ma_p^*(\mu,\nu)$ may not be reduced to a single value. Moreover it is not constant as a function of $p\in ]0,1[$. See Example \ref{ex:main} about these two facts. See also Lemma \ref{lem:coincide} the maximal amount of mass that can stay must also stay on place.
\item Our paper being concerned with the value $p=1$ and its limits $p\to  1^\pm$, we interrupt here the description of the power costs. It is continued in Remarks \ref{rem:zero} and \ref{rem:coulomb} where we give an account on $p=0$ and, beyond, $p<0$.
\end{itemize}

The novelty of this paper is that for the critical and historical parameter $p=1$ we distinguish a special element of $\ma_p^*(\mu,\nu)$ that we call the \emph{excursion coupling}. This element may be seen as a solution for $p\to1^-$, the counterpart of the quantile coupling that would be the solution for $p\to1^+$.

\begin{mthem}\label{main}
Let $\mu$ and $\nu$ be two probability mesures in $\p_1(\R)$, $\pi\in\ma(\mu,\nu)$ and $q\in]0,1[$. The following assumptions are equivalent.
\begin{enumerate}
\item\label{un} (Solution of the $L^{1^-}$ limit transport problem) There exists a sequence $(\pi_n,p_n)_n\in \N$ with $p_n<1$ and $\pi_n\in\ma^*_{p_n}(\mu,\nu)$ for every $n$, such that $(\pi_n,p_n)\to_n (\pi,1)$.
\item\label{deux} (Solution of the $L^{1,q}$ secondary transport problem) $\pi$ is in $\ma^*_1(\mu,\nu)$ and in this space it minimizes $\gamma\in \ma^*_1(\mu,\nu)\mapsto \iint |y-x|^{q}\,\d\pi(x,y)$.
\item\label{trois} (Monotonicity) $\pi$ is concentrated on a set $S\subset \R\times \R$ whose arches do not cross, do not connect and have the same orientation when they are nested (see Definition \ref{def:arches} and Figure \ref{fig:forbidden}).
\item\label{quatre} (Excursion coupling) $\pi$ is the excursion coupling of $\mu$ and $\nu$ as defined in Section \ref{sec:defexc}.
\end{enumerate}
\end{mthem}

In Section \ref{sec:proof} we will prove the Main Theorem as well as the following corollary based on the uniqueness of a coupling satisfying \emph{\ref{un}} or \emph{\ref{deux}} in the Main Theorem.

\begin{cor}\label{cor:main_cor}
Let $\mu$ and $\nu$ be two probability mesures in $\p_1(\R)$.  The excursion coupling $\pi\in\ma(\mu,\nu)$ satisfies the two following statements.
\begin{itemize}
\item[1'.] For $(\pi_n,p_n)_{n\in \N}$ satisfying $p_n\to 1$ (with $p_n<1$) and $\pi_n\in\ma^*_{p_n}(\mu,\nu)$ it holds $\pi_n\to \pi$.
\item[2'.] For every $q'<1$, the map $\gamma\in \ma^*_1(\mu,\nu)\mapsto \iint |y-x|^{q'}\,\d\pi(x,y)$ is minimized by $\pi$.
\end{itemize}
%
\end{cor}

The paper is organized as follow: To give the Main Theorem a complete meaning we first briefly make \emph{\ref{trois}.}\ more precise in Definition \ref{def:arches} and define in Section \ref{sec:defexc} the excursion coupling attached to a pair $(\mu,\nu)$. Note that this definition relies on several facts concerning functions with bounded variations that will be recalled and established in \S\ref{ssec:tv}. Then we prove step by step the following implications: We prove $\emph{\ref{un}}\Rightarrow \emph{\ref{trois}}$ and $\emph{\ref{deux}}\Rightarrow \emph{\ref{trois}}$ in Section \ref{sec:1to3}. We prove $\emph{\ref{trois}}\Rightarrow \emph{\ref{quatre}}$ in Section \ref{sec:3to4}. The fact that there exists at least one solution to the $L^{1^-}$ problem and the $L^{1,q}$ secondary problem completes the proof (see Section \ref{sec:proof}). In Section \ref{sec:remarks} we provide some more comments.

\begin{defi}[Arches in the Main Theorem]\label{def:arches}
Let $S$ be a subset $\R\times \R$, whose elements we call \emph{transport routes}. In what follows $[a,b]$ means $[\min(a,b),\max(a,b)]$ and the routes are seen as arches (half-circles) over the real line.

\begin{itemize}
\item The routes of $S$ are said \emph{non-crossing arches} if for every $(x,y)\in S$ and $(x',y')\in S$ at least one of the three happens.\begin{itemize}
\item $[x,y]\cap [x',y']=\emptyset$
\item $[x,y]\cap [x',y']=\{z\}$ for some $z\in \R$
\item One of the two, $[x,y]$ or $[x',y']$, is included in the other.
\end{itemize}
\item The arches \emph{do not connect} means that if $(x,y)\in S$, $(x',y')\in S$ and $\min(|y-x|,|y'-x'|)>0$ then $y\neq x'$. 
\item The \emph{nested arches have the same orientation} if for every $(x,y)\in S$ and $(x',y')\in S$ such that $[x',y']\subset ]x,y[$ we have $(y-x)(y'-x')\geq 0$.
\end{itemize}
\end{defi}

The first property can be summed up saying that in the plane the two circles of diameter $|y-x|$ and $|y'-x'|$ linking $x$ and $y$, and $x'$ and $y'$, respectively, do not cross. The second property states that a point can not be at the same time a starting and arriving point. In the last property is stated that transporting $x$ to $y$, and $x'$ to $y'$ must be done in the same direction when one of the arches is included in the other. See Figure \ref{fig:forbidden} for a representation of the forbidden configuration and the authorized rerouting -- where $(x,y')$ and $(x',y)$ are the new routes.

\noindent

\emph{Aknowledgements.} I warmly thank Guillaume Carlier, Augusto Gerolin and Martin Huesmann for discussions, bibliographic or editorial suggestions.

\section{Definition of the excursion coupling}\label{sec:defexc}


Given two probability measures $\mu$ and $\nu$, we consider the signed measure $\sigma=\mu-\nu$ and its cumulative distribution function
\begin{align}
F_\sigma:x\in\R\mapsto \sigma]-\infty,x]=F_\mu(x)-F_\nu(x).
\end{align}
Since it is the difference of $F_\mu$ and $F_\nu$, $F_\sigma$ is c\`adl\`ag (right continuous and with left limits at any point) and has limits $0$ in $\pm\infty$. The graph $\graph(F_\sigma)=\{(x,F_\sigma(x))\in \R^2:\,x\in\R)\}$ of $F_\sigma$ can be completed at discontinuity points $x$ by vertical segments $[(x,F_\sigma(x^-)),(x,F_\sigma(x))]\subset \R^2$ where $F_\sigma(x^-)=\lim\limits_{\eps\to 0,\, \eps>0}F_\sigma(x-\eps)$ denotes the left limit at $x$. We denote by $F^*_\sigma$ the multivalued map defined by $F^*_\sigma(x)=[F_\sigma(x^-),F_\sigma(x)]$ at discontinuity points $x$ and $F^*_\sigma(x)=\{F_\sigma(x)\}$ at continuity points. Its graph is 
\[\graph(F^*_\sigma)=\{(x,h)\in \R^2:\,h\in F^*_\sigma(x)\}\]
and we may also denote it by $\graph^*(F_\sigma)$. If $(x,h)\in \graph(F^*_\sigma)$ the real $x$ is called a generalized solution of $F_\sigma=h$.

We further introduce the following subsets of $\graph^*(F_\sigma)$:
$\graph^{*,+}(F_\sigma)$ is the set of increasing points $(x,h)$, i.e such that in a neighborhood $U_x$ of $x$ any point $(x',h')\in \graph^*(F_\sigma)$ with $x'\in U_x\setminus \{x\}$ satisfies $(h'-h)(x'-x)>0$.
$\graph^{*,-}(F_\sigma)$ is the set of decreasing points $(x,y)$, i.e such that in a neighborhood $U_x$ of $x$ any point $(x',h')\in \graph^*(F_\sigma)$ with $x'\in U_x\setminus \{x\}$ satisfies $(h'-h)(x'-x)<0$.

%
%
\begin{them}[Excursion couplings can be defined]\label{thm:construction}
Let $\mu$ and $\nu$ be probability measures and $\sigma=\mu-\nu$, $F_\sigma$ and $\graph(F^*_\sigma)$ be defined as above. One can define a transport of $\ma(\mu,\nu)$ as described in what follows, the results implicitly stated during this construction (see Remark \ref{rem:implicit} for a list) are correct and we call \emph{excursion coupling} the resulting coupling.

We distinguish two cases (the first one is a special case of the second).
\begin{enumerate}
\item Assume that $\mu$ and $\nu$ are singular measures ($\mu\perp \nu$). We define a coupling $(X,Y)$ where $X\sim \mu$ and $Y\sim \nu$. Let $\theta$ be the measure with density $h\mapsto i^\pm(h)=\card[(\R\times\{h\})\cap \graph^{*,+}(F_\sigma)]+\card[(\R\times\{h\})\cap \graph^{*,-}(F_\sigma)]$. Let $H$ be a random variable with law $\theta/2$ and note
$(\R\times\{H\})\cap \graph^{*}(F_\sigma)= \{x_1,\ldots,x_N\}\times \{H\}$ where $N=i^\pm(H)$ and $x_1<\cdots< x_N$. Conditionally on $H$ and $H>0$, the random vector $(X,Y)$ is defined to be uniform on
$\{(x_1,x_2),(x_3,x_4),\ldots,(x_{N-1},x_N)\}$. Conditionally on $H$ and $H<0$ it is uniform on $\{(x_2,x_1),(x_4,x_3),\ldots,(x_{N},x_{N-1})\}$.

\item If $\mu=\eta+\mu_0$ and $\nu=\eta+\nu_0$ with $\mu_0\perp \nu_0$, with probability $\eta(\R)$ the random vector $(X,Y)$ satisfies $X=Y$ and $X\sim \eta$. On the complementary event, with probability $1-\eta(\R)=\mu_0(\R)$ it is distributed as the  coupling of the singular measures $\mu_0(\R)^{-1}\mu_0$ and $\nu_0(\R)^{-1}\nu_0$ defined in the first item.
\end{enumerate}
\end{them}

\begin{rem}\label{rem:implicit}
To make Theorem \ref{thm:construction} a rigorous definition of the excursion coupling we will have to prove that $\theta/2(\d h)$ is a probability density, that $N$ is almost surely both finite and even with $(\R\times\{h\})\cap \graph^{*}(F_\sigma)=(\R\times\{h\})\cap (\graph^{*,-}(F_\sigma)\cup\graph^{*,+}(F_\sigma))$. Moreover, we must prove that the laws of $X$ and $Y$ are $\mu$ and $\nu$, respectively.
\end{rem}

\section{Monotonicity of the solutions of the $L^{1,q}$ and $L^{1^-}$ transport problems}\label{sec:1to3}

In this section we prove the implications $\emph{\ref{un}}\Rightarrow \emph{\ref{trois}}$ and $\emph{\ref{deux}}\Rightarrow \emph{\ref{trois}}$ of the Main Theorem. The name ``(cyclical-)monotonicity'' in the title is a generic name in Optimal Transport that in this section is represented by property $\emph{\ref{trois}}$. (Cyclical-)monotonicity results are variations of the following simple swapping lemma that concerns the $L^p$ transport problem for cycles of length two. For the $L^{1^-}$ transport problem we will firstly interpret Lemma \ref{lem:swap1} for $p<1$ and secondly let $p$ go to $1$. For the $L^{1,q}$ transport problem will need a result analogue to \ref{lem:swap1} but more specific result: It will be Lemma \ref{lem:swap2} on page \pageref{lem:swap2}. 


\begin{lem}[Swapping lemma]\label{lem:swap1}
Let $p$ be positive and $\pi$ be in $\ma_p^*(\mu,\nu)$. Assume moreover $T_p(\pi)<+\infty$. Consider a set $S\subset \R^2$ such that $\pi(S)=1$ and for $(x,y)\in S$ any neighborhood of $(x,y)$ has positive measure (Notice that the support of $\pi$ satisfies these conditions, so that we may choose $S=\spt(\pi)$). For any $(x,y)$ and $(x',y')$ in $S$, the following holds:
\begin{align}\label{eq:swap}
|y-x|^p+|y'-x'|^p\leq |y-x'|^p+|y'-x|^p.
\end{align}
\end{lem}
\begin{proof}
Striking for a contradiction, suppose that the opposite identity holds. Then for some $(x,y),(x',y')\in S$ there exists $\eps>0$ such that for every $(a,b,a',b')$ with $\max(|a-x|,|a'-x'|,|b-y|,|b'-y'|)\leq \eps$ one has $|b-a'|^p+|b'-a|^p<|b-a|^p+|b'-a'|^p$. The two balls (in the $\infty$-norm) of radius $\eps$ centered in $(x,y)$ and $(x',y')$ have positive measure. We choose $\eps$ small enough to make their intersection the empty set. Then it is easily possible to replace $\pi$ by a competitor $\pi'$ that coincide with it outside the balls $\mathcal{B}_\infty((a,b),\eps)$, $\mathcal{B}_\infty((a',b'),\eps)$, $\mathcal{B}_\infty((a,b'),\eps)$ and $\mathcal{B}_\infty((a',b),\eps)$, has marginals $\mu$ and $\nu$ and satisfies $T_p(\pi')<T_p(\pi)$ (see e.g \cite[page 129]{GMcC}), a contradiction. 
\end{proof}

This simple principle, allows for a complete characterization of $\pi$ in the case $p>1$. We provide it now for the sake of completeness and for further comparison with the $L^{1^-}$ limit and $L^{1,q}$ secondary problems.
\paragraph{The $L^p$ transport problem for $p>1$}

\begin{figure}[h!!]
\parbox{6cm}{
\begin{flushleft}
   \def\svgwidth{6cm}
      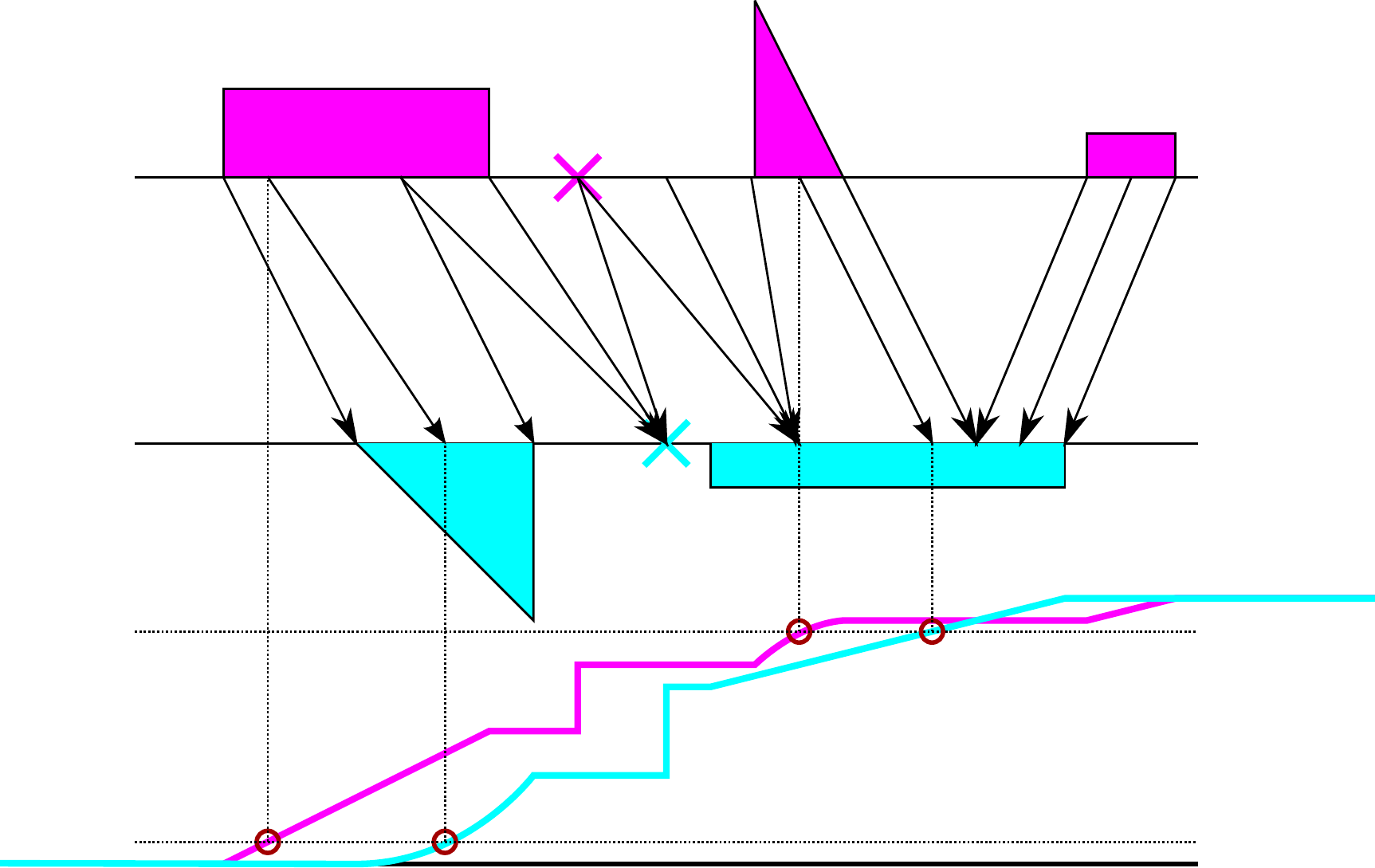
\end{flushleft}
}
\hspace{.5cm}
\parbox{6cm}{
\begin{flushright}
   \def\svgwidth{6cm}
   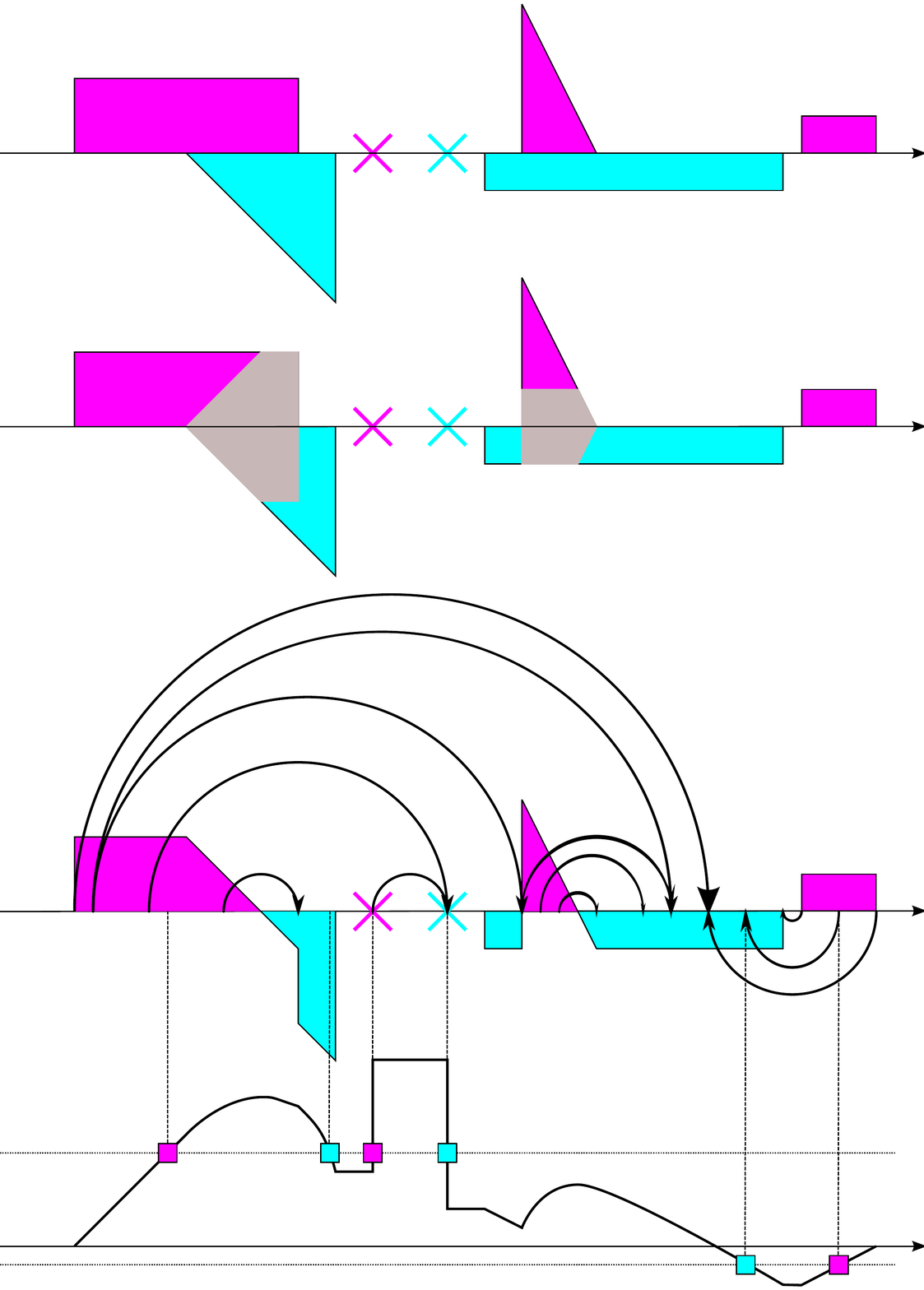
\end{flushright}
}
\caption{Left: the quantile transport is based on simultaneous simulation through the cumulative distribution functions $F_\mu$ and $F_\nu$. The dashed horizontal lines (below on the figure) are chosen uniformly in $]0,1[$. The transport routes are materialized by arrows (above on the figure) that do not cross.}
\label{fig:left}
\caption{Right: the moving mass of the excursion coupling is simulated through the intersections with $F_\nu-F_\nu$ ; the commun mass $\mu\wedge \nu$ stay on the on the same place. Given the horizontal line the excursion is chosen uniformly (among two or one pairs for the lines on the figure). The increasing intersection corresponds to $\mu$ and the decreasing one to $\nu$. The transport routes are materialized by arches that do not cross.}
\end{figure}

In this case the identity $|y-x|^p+|y'-x'|^p\leq |y-x'|^p+|y'-x|^p$ obtained by applying Lemma \ref{lem:swap1} is equivalent to $(y'-y)(x'-x)\geq 0$. This may be graphically represented by the condition that segments in $\R^2$ connecting for every $(x,y)\in S$ the point $(x,1)$ to $(y,0)$ are not allowed to cross each other, see Figure \ref{fig:left}. These segments may be interpreted as transport routes between $\mu$ concentrated on the line $\{y=1\}$ and $\nu$ concentrated on the $x$-axis $\{y=0\}$. The striking fact is that, given $\mu$ and $\nu$, the elements $\pi$ of $\ma(\mu,\nu)$ being concentrated on such a set $S$ are in fact reduced to a single transport plan, called the \emph{quantile transport plan}. The latter is the law of $(G_\mu,G_\nu)$ on the probability space $(]0,1[,\lambda)$ of quantiles, where $\lambda$ is the Lebesgue measure and for any a real probability measure $\eta$, the quantile function $G_\eta$ is defined as a pseudo-inverse of $F_\eta:x\mapsto\eta(]-\infty,x])$, namely
\[G_\eta(\alpha)=\inf\{x\in \R:\,F_\eta(x)\geq \alpha\}\]
(this infimum is a minimum).

\subsection{The $L^p$ transport problem for $p<1$ and the $L^{1^-}$ limit transport problem}\label{ss:limit}

\paragraph{Important preliminary comparaison to the convex case $p>1$}
In the previous section we recalled that for $p>1$, provided the $L^p$ transport problem admits a solution $\pi_p$ with $T_p(\pi_p)<\infty$, this trasnport plan $\pi_p$ is the \emph{unique solution} and it is the quantile coupling. In particular it is \emph{independent of the value of $p$}.

The two last assumptions are \emph{both false} in the case $p<1$, which we prove in the following example.

\begin{ex}\label{ex:main}
Set $\mu=\frac12(\delta_0+\delta_5)$ and $\nu=\frac12(\delta_4+\delta_9)$. The set of transport plans is easily described as
\[\ma(\mu,\nu)=\{\pi^\lambda=\lambda\pi''+(1-\lambda)\pi':\ \lambda\in [0,1]\}\]
where $\pi'=\frac12(\delta_{0,4}+\delta_{5,9})$ and $\pi''=\frac12(\delta_{0,9}+\delta_{5,4})$. In the case $p=1/2$ every transport plan $\pi$ gives the same global cost $T_p(\pi^\lambda)=\lambda\frac12(\sqrt4+\sqrt4)+(1-\lambda)\frac12(\sqrt 9+\sqrt 1)=2$. This proves the \emph{non-uniqueness}; $\ma_{1/2}^*(\mu,\nu)=\ma(\mu,\nu)$. Moreover, with the same marginals, depending whether $p<1/2$ or $p>1/2$ the measure $\pi'=\pi^0$ or $\pi''=\pi^1$, respectively, is the unique optimal transport plan. Hence the solution \emph{depends on the value of $p$}.
\end{ex}

\begin{figure}\label{fig:example}
\begin{center}
   \def\svgwidth{10cm}
   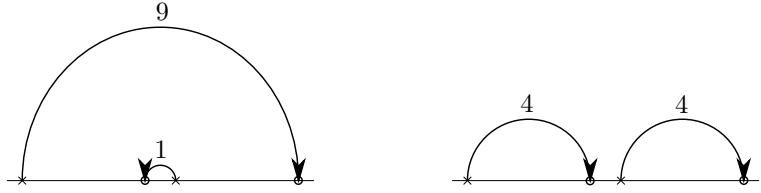
\caption{For transporting the two crosses to the two circles, each symbol being of mass $1/2$, the left pattern is optimal for $p\leq1/2$, the right is for $p\geq1/2$.}
\end{center}
\end{figure}

In the two next paragraphs we prove that in the $L^p$ transport problem (for $p<1$) arches do not connect and do not cross. In the third next paragraph this will be transmitted to the $L^{1^-}$ (limit) transport problem where we also prove that nested arches have the same orientation, completing the proof of $\emph{\ref{un}}\Rightarrow \emph{\ref{trois}}$ in the Main Theorem.

\paragraph{Conclusion concerning coinciding points for $p<1$}

Let $S$ satisfying \eqref{eq:swap} in the swapping lemma, Lemma \ref{lem:swap1} for $p<1$ and $(x,y,x',y')$ be in $S\times S$. Suppose moreover that $y$ and $x'$ coincide meaning that one can take a route from $x$ to $y$ and a second one from $y=x'$ to $y'$. Then studying the variations of $t\in [0,+\infty[\mapsto (t+h)^p-t^p$ and considering $h=|y-y'|$ we see that this can only occur in one of the degenerate case $x=y$ or $x'=y'$ -- compare with the two most right subfigures in Figure \ref{fig:forbidden}. With respect to the terminology of Definition \ref{def:arches} we have proved that a solution $\pi$ of the $L^p$ transport problem $(p<1)$ is concentrated on a set whose \emph{arches do not connect}. This well-known situation has been studied be Gangbo and McCann in \cite[Proposition 2.9]{GMcC} more than 20 years ago. It yields that $\pi$ can be decomposed as the sum $\pi=\pi_\Delta+\pi_0$ where $\pi_\Delta=(\id\times\id)_\#(\mu\wedge \nu)$. Let us remind the argument for the sake of completeness.

\begin{lem}  \label{lem:coincide}
If $\pi\in \ma(\mu,\nu)$ is concentrated on a set $S$ whose arches do not connect, it can be decomposed as follows: $\pi=\pi_\Delta+\pi_0$ where $\pi_\Delta=(\id\times\id)_\#(\mu\wedge \nu)$ and, for $\Delta=\{(x,y)\in \R^2:\,y=x\}$ denoting the diagonal, it holds $\pi_0(\Delta)=0$. Consequently, $\pi_0$ is a transport plan with the two marginals singular with respect to each other. 
\end{lem}
\begin{proof} Let us write $\pi=\pi_\Delta+\pi_0$ where $\pi_0$ is concentrated on $S\setminus \Delta$ and $\pi_\Delta$ is concentrated on $\Delta$. Therefore $\pi_\Delta$ writes $(\id\times\id)_\#\eta$ and the marginals of $\pi_0$ are $\mu-\eta$ and $\nu-\eta$. They are respectively concentrated on the two projections $\{x'\in \R:\,\exists y'\neq x',\,(x',y')\in S\setminus \Delta\}$ and $\{y\in \R:\,\exists x\neq y,\,(x,y)\in S\setminus \Delta\}$ of the set $S\setminus \Delta$. The fact that these two sets do not intersect is the direct consequence of our assumption. It follows $\mu-\eta\perp\nu-\eta$ which is equivalent to $\eta=\mu\wedge \nu$.
\end{proof}


\paragraph{Interpretation of the swapping lemma for $p<1$ and non coinciding points}
For $p<1$ the use of the swapping lemma furnishes, compared to $(y'-y)(x'-x)\geq 0$, less direct information. Equation \eqref{eq:swap} may be seen, similarly as in Example \ref{ex:main}, as a competition of two transport plans, each transporting two points in two other points, in one or the other way. For this reduced transport problem if $|y-x|^p+|y'-x'|^p\neq |y'-x|^p+|y-x'|^p$ at most one of the two is true $(x,y,x',y')\in S\times S$ or $(x,y',x',y)\in S \times S$. Unlike the situation studied for $p>1$, to determine which transport is better it does not only depend on the relative positions of $x$ and $y$ with respect to $x'$ and $y'$, respectively, but on the $4!$ relative positions of the four points. Moreover, even though this ranking of the four point is important and yields the conclusion in configuration $xxyy$ and $xyyx$ we know since Example \ref{ex:main} and Figure \ref{fig:example} that it does not permit to conclude in configuration $xyxy$. 

\begin{notation} The notation $xyxy$ denotes the configurations of two routes $(x,y)$ and $(x',y')$ where $x<y<x'<y'$ or $x<y'<x'<y$ or $x'<y<x<y'$ or $x'<y'<x<y$, the different alternatives being signed as $xyx'y'$, $xy'x'y$, $xy'xy'$ and $x'yxy'$, respectively. 
\end{notation}

Since the $L^p$ problem in $\ma(\mu,\nu)$ is in bijection with the one in $\ma(\nu,\mu)$, when we study configuration $xyxy$ we in fact also study $yxyx$. In order to consider all configurations without coinciding points we finally only have to look at $xxyy$, $xyyx$ and $xyxy$. Here is the conclusion for these three cases. They are also illustrated on Figure \ref{fig:forbidden} (corresponding, in the same order, to the first three pairs of patterns, from the left).

\begin{itemize}\label{list:lp}
\item In the first case $xx'y'y$ is allowed and $xx'yy'$ forbidden. To see that, one can study the variations of $t\in [0,+\infty[\mapsto (t+h)^p-t^p$ where $h>0$ and apply it to $h=|y-y'|$.
\item In the second case $xyy'x'$ is allowed and $xy'yx'$ forbidden. This is a simple consequence of the fact that $d\in \R^+\mapsto d^p$ is increasing.
\item In the last case, as observed in Example \ref{ex:main}, it depends on $p$ whether $xyx'y'$ or $xy'x'y$ is forbidden or authorized.
\end{itemize}
With the two first points we have proved that the arches of $S$ \emph{do not cross}. 


\paragraph{From the $L^p$ to the $L^{1^-}$ transport problem} 
In this paragraph we prove that the solutions of the $L^{1^-}$ transport problem have arches that do not connect, do not cross and have the same orientation when they are nested. This is implication $\emph{1}\Rightarrow \emph{3}$ of the Main Theorem.

Consider $\pi$ such that there exists a sequence $(\pi_n)_{n\in \N}$ weakly converging to $\pi$ where $\pi_n\in \ma_{p_n}^*(\mu,\nu)$ for $p_n\to 1^-$. It also holds $\pi_n\otimes\pi_n\to \pi\otimes \pi$, actually an equivalent fact. Hence, if $F\subset \R^4$ is a closed set and $\pi_n\otimes \pi_n(F)=1$, the equation goes to the limit. In particular this holds for the complementary set $F_1$ of the open set $\{(x,y,x',y')\in\R^4:\, x<x'<y<y'\text{ or }x'<x<y'<y\}$ that encodes the condition on non-intersecting arches (first condition of Definition \ref{def:arches}). As it is satisfied by $\pi_n$ it is also satisfied by $\pi$. Therefore $\pi$ is concentrated on a set whose arches \emph{do not cross}.

%

Suppose by contradiction that the set $F^c_3\subset \R^4$ of pairs of nested arches that do not have the same orientation has positive measure for $\pi\otimes \pi$. Due to the countable additivity of $\pi\otimes \pi$ there exists rational numbers $x<y'<x'<y$ and $\eps\in \mathbb{Q}^+$ with $\min(|y'-x|,|x'-y'|,|y-x'|)>\eps$ such that $U=]x\pm\eps/10[\times ]y\pm \eps/10[\times ]x'\pm\eps/10[\times ]y'\pm \eps/10[$ has positive measure. Since 
\[[(y'-x)+2\eps/10]^{p_n}+[(y-x')+2\eps/10]^{p_n}\leq [(y-x)-2\eps/10]^{p_n}+[(x'-y')-2\eps/10]^{p_n}\]
for $p_n$ close enough to 1, the open set $U$ has measure zero for the corresponding measures $\pi_n\otimes \pi_n$. Therefore, it has measure zero for $\pi\otimes\pi$ as well, a contradiction. We conclude that $\pi$ concentrated on a set whose \emph{nested arches have the same orientation}.


The implication $\emph{\ref{un}}\Rightarrow\emph{\ref{trois}}$ finally amounts to prove that $\pi$ is concentrated on a set satisfying the second condition of Definition \ref{def:arches}: \emph{the arches do not connect}. Recall from Lemma \pageref{lem:coincide} that $\pi_n=(\id\times \id)_\#(\mu\wedge \nu)+\pi_n^0$ where $\pi^0_n$ has marginals $\mu^0=\mu-(\mu\wedge \nu)$, $\nu^0=\nu-(\mu\wedge \nu)$ and $\mu^0\perp \nu^0$. Thus $\pi$ is the limit of the sequence if and only if it can be written $(\id\times \id)_\#(\mu\wedge \nu)+\pi^0$ where $\pi^0$ is the limit of $(\pi^0_n)_{n\in \N}$. Therefore, $\pi^0$ has marginals $\mu_0$ and $\nu_0$. Thus there exists $A$ and $B=A^c$ such that $\pi$ is concentrated on $E=[(A\times \R)\cap(\R \times B)]\cup \Delta$. This exactly means that it is concentrated on a set whose arches do not connect.
%
%
%
%
%

\subsection{The $L^{1,q}$ secondary transport problem}\label{ss:second}

\paragraph{Swapping lemma for the $L^{1,q}$ secondary transport problem and interpretation}
We furnish now a swapping lemma corresponding to the $L^{1,q}$ transport problem.

\begin{lem}[Swapping lemma for the secondary problem]\label{lem:swap2}
Let $\pi$ be a transport plan such that $T_1(\pi)<+\infty$ and $\pi\in \ma_{1,q}^{**}(\mu,\nu)$ for some fixed $q<1$. There exists a set $S\subset \R^2$ with $\pi(S)=1$ such that for any $(x,y)$ and $(x',y')$ in $S$ it holds
\begin{align}\label{eq:swap0}
|y-x|+|y'-x'|\leq |y-x'|+|y'-x|\
\end{align}
and if $|y-x|+|y'-x'|=|y-x'|+|y-x'|$, the following holds:
\begin{align}\label{eq:swap1}
|y-x|^q+|y'-x'|^q\leq |y-x'|^q+|y'-x|^q.
\end{align}
\end{lem}
\begin{proof}

Let $\pi$ be an optimal transport plan for the $L^1$ primary and the $L^{1,q}$ secondary transport problem. Let $N_{y,-}$ be  $\{(x,y)\in \spt(\pi):\,\exists \eps>0, \pi([x-\eps,x+\eps]\times[y-\eps,y])=0\}$. It is the set of routes in the support of $\pi$ such that at some size $\eps>0$ no mass is transported from the ball of center $x$ and radius $\eps$ to the left of $y$ at distance less than $\eps$. We have $\pi(N_{y,-})=0$. The proof of this result is postponed to Lemma \ref{lem:postponed}. The symetricaly defined sets $N_{y,+}$, $N_{x,-}$ and $N_{x,+}$ have measure zero too (for instance $N_{x,+}$ here denotes the set $\{(x,y)\in \spt(\pi):\,\exists \eps>0, \pi([x,x+\eps]\times[y-\eps,y+\eps])=0\}$). We define now $S$ as $S_0\setminus (N_{y,-} \cup N_{y,+} \cup N_{x,-} \cup N_{x,+})$ where $S_0$ is any set, as for instance $\spt(\pi)$, satisfying the condition in Lemma \ref{lem:swap1}. Observe that $\pi(S)=1$.

We are now ready for the proof. It is almost identical to that of Lemma \ref{lem:swap1} that we invite the reader to read again: the principle is that $\pi$ is compared to a competitor $\pi'$ defined rerouting part of the mass around $(x,y)$ and $(x',y')$ to mass aroud $(x,y')$ and $(x',y)$. From Lemma \ref{lem:swap1} we note that \eqref{eq:swap0} is satisfied for any $(x,y,x',y')\in S\times S$. Aiming for a contradiction, suppose that  for some routes $(x,y)$ and $(x',y')$ in $S$ it holds at the same time $|y-x|+|y'-x'|= |y-x'|+|y'-x|$ and $|y'-x|^q+|y-x'|^q< |y-x|^q+|y-x|^q$. Until the end of the paragraph let us see that without loss of generality we can assume $x<x'\leq y<y'$: The first equation induces `$\max(x,x')\leq \min(y,y')$ or $\max(y,y')\leq \min(x,x')$'. Without loss of generality we can assume the first case. The second inequality implies $x\neq x'$ and $y\neq y'$ and without loss of generality we assume $x<x'$. This finally implies $y<y'$.

Apparently the swapping method used in the proof of Lemma \ref{lem:swap1} that consists in picking mass in equal quantity around the points $x$, $x'$, $y$ and $y'$ can be applied without problem providing a competitor $\pi'\in \ma(\mu,\nu)$ with $T_q(\pi')<T_q(\pi)$, a contradiction. However, this argument only works if $x'<y$ and not directly if $x'=y$ because $\pi'\in\ma^*_1(\mu,\nu)$ can not be certified: during the swap the routes $(x,y)$ and $(x',y')$ with $x<x'=y<y'$ have a priori in their neighborhood routes $(\tilde{x},\tilde{y})$ and $(\tilde{x}',\tilde{y}')$ with $\tilde{x}<\tilde{y}<\tilde{x}'<\tilde{y}'$ that do no longer satisfy \eqref{eq:swap0}, so that $T_1(\pi')>T_1(\pi)$ is made possible. The relation $T_q(\pi')<T_q(\pi)$ is no longer a contradiction because $\pi'\notin \ma^*_1(\mu,\nu)$. In this critical situation let us call $z$ the point $x'=y$. As $(x,y)$ and $(x',y')$ are not in $N_{y,-} \cup N_{y,+} \cup N_{x,-} \cup N_{x,+}$ we can swap selecting the mass on the proper side of $z$. More precisely for $\pi$ there is some mass traveling from a neighborhood (as small as we want) of $x$ to a small right neighborhood of $z$. There is the same mass on a small left neighborhood of $z$ transported to a small neighborhood of $y'$. We swap for defining $\pi'$: the mass around $x$ is transported around $y'$ and the mass directly on the left of $z$ is transported directly to the right of $z$. We obtain $T_q(\pi')<T_q(\pi)$ and keep  $T_1(\pi')=T_1(\pi)$.

\end{proof}

\begin{lem}\label{lem:postponed}
Let $\pi$ be a solution to the $L^{1,q}$ secondary optimal transport and $N_{y,-}$ be  $\{(x,y)\in \spt(\pi):\,\exists \eps>0, \pi([x-\eps,x+\eps]\times[y-\eps,y])=0\}$. Then $\pi(N_{y,-})=0$.
\end{lem}
\begin{proof}
Fix $a<b$ two rational numbers. Consider the set $A_{a,b}$ of points $(x,y)\in \spt(\pi|_{]a,b[\times \R})$ such that there exists $\eps>0$ with $\pi([x-\eps,x+\eps]\times[y-\eps,y])=0$ and moreover $x\in]a,b[\subset [x-\eps,x+\eps]$. Then every point $y\in \pr^2(A_{a,b})$ is in the support of the measure $\nu_{a,b}$ defined by $(\pr^2)_\#\pi|_{]a,b[\times \R}$, but it is also isolated on the left in the sense that $\nu_{a,b}([y-h,y])=0$ for $h$ small enough. One can easily convince that there are countably many such points $y$ and they are not atoms. Therefore, $\nu(\pr^2(A_{a,b}))=0$ and $\pi(A_{a,b})=\pi|_{]a,b[\times \R}(A_{a,b})\leq \nu(\pr^y(A_{a,b}))=0$. Finally, since $N_{y,-}=\bigcup_{a<b\in \mathbb{Q}} A_{a,b}$ we find $\pi(N_{y,-})=0$.
\end{proof}

\begin{figure}
\begin{center}
   \def\svgwidth{12cm}
   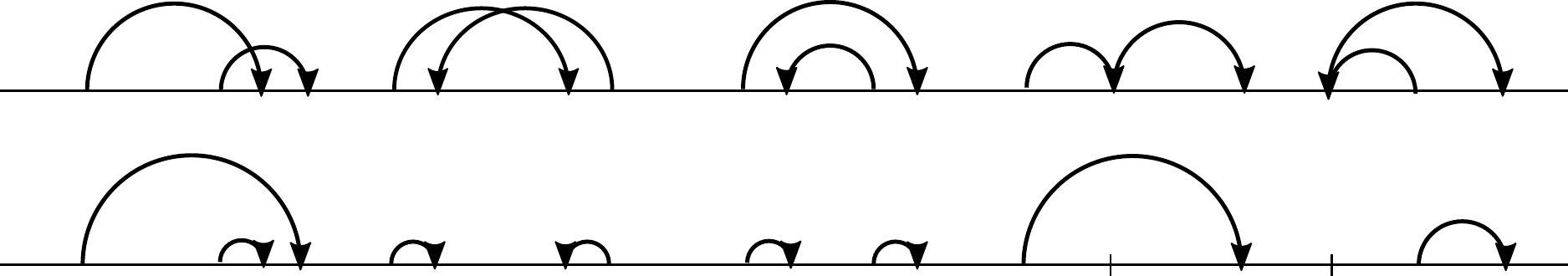
\caption{Forbidden (first line) versus authorized (second line) configurations in the $L^{1,q}$ secondary transport problem, where $p<1$.} \label{fig:forbidden}
\end{center}
\end{figure}

We can now give the geometric meaning for the routes of set $S$ as in Lemma \ref{lem:swap2}. As for the $L^{1^-}$ problem we compare the different configurations with their alternative -- recall the list on page \pageref{list:lp} -- and again these comparaisons are illustrated on the three most left patterns of Figure \ref{fig:forbidden}.

\begin{itemize}
\item  Pattern $xx'y'y$ is allowed and $xx'yy'$ forbidden. To see that, since $|y'-x'|+|y-x|=|y'-x|+|y-x'|$ we have to look at the secondary problem. 
\item Pattern $xyy'x'$ is allowed and $xy'yx'$ forbidden because as in the $L^p$ problem, \eqref{eq:swap0} is a strict inequality.
\item Finally, pattern $xyx'y'$ is allowed and $xy'x'y$ is forbidden because as in the $L^p$ problem, \eqref{eq:swap0} is a strict inequality.
\end{itemize}


Some of the points $x,y,x',y'$ may be equal. Swapping does not change the cost if $x=x'$ or $y=y'$. We have only to look at $x=y'$ and $x'\neq y$ and see that it is never better than $x=y$ and $x'\neq y'$ (see on Figure \ref{fig:forbidden} the two patterns on the right).
\begin{itemize}
\item For the pattern $x(yx)y$, meaning $x<y=x<y$, we have $|y'-x'|+|y-x|=|y'-x|+|y-x'|$ but the secondary problem tells us to choose $xx'y'y$ in place of $xyx'y'$.
\item For $(xy)xy$, from the primary problem we choose $x=y<x'<y'$ in place of $x=y'<x'<y$.
\end{itemize}

Finally as for the $L^{1^-}$ limit transport problem, we have proved that if $\pi$ is a solution of the $L^{1,q}$ secondary transport problem it is concentrated on a set $S$ whose arches do not cross, do not connected, and have the same orientation when they are nested.

\section{Transport plans concentrated on monotone sets of arches are the excursion coupling}\label{sec:3to4}

\subsection{Proof of Theorem \ref{thm:construction} defining the excursion coupling}\label{ssec:tv}

We need to explain why Theorem \ref{thm:construction} can define the excursion coupling. We only need to investigate the construction in case \emph{1} where $\mu$ and $\nu$ are singular ($\mu\perp \nu$), which we thus assume in the present subsection.

With the following lemma we will be able to handle with $F_\sigma$ as if it were a continuous function.


\begin{lem}[Generalized intermediate value theorem]\label{lem:tvi}
For any c\`adl\`ag function $F$, any $x_0,x_1\in \R$ and $h\in \R$ such that $x_0<x_1$ and  $(F(x_0)-h)(F(x_1)-h)<0$, there exists $x\in]x_0,x_1]$ such that $(x,h)\in \graph^*(F)$. 
\end{lem}
\begin{proof}
Let $F$, $x_0,x_1$ be as in the statement. Without loss of generality we assume $h=0$, $F(x_0)<0$ and $F(x_1)>0$. Let $x$ be the infimum of $A=\{x\in[x_0,x_1]:\,F(x)\geq 0\}$. As $A\ni x_1$ it is a not empty set. As moreover $F$ is right continuous we have $x\in A$ and $x\neq x_0$. Due to the definitions of $A$ and $x$, any $x'<x$ satisfies $F(x')<0$.  If $F$ is left-continuous at $x$ we have $F(x)=0$. If it is not, as $F(x-)<0$ we also have $(x,0)\in\graph(F^*)$.
\end{proof}

%

A result by Bertoin and Yor establishes a relation between the occupation measure in a set $B\subset \R$ of a function $F$ of finite variation and its variations when its values are in $B$, see Remark \ref{rem:BY} for details. In particular we can apply their Theorem 1 in \cite{BY} (see also their \S 5) to $F=F_\sigma$ and $B=\R$, and relate the total variation (without its saltus part) with the number of solutions of the equations $F_\sigma=h$, where $h$ goes over $\R$: For any points $s\leq t$ in $\overline{\R}$

\begin{align} \label{eq:bla}
TV_s^t(F_\sigma)-\sum_{x\in ]s,t]}|\sigma(x)|&=\int_\R\card\{x\in [s,t]:\,(x,h)\in\graph(F_\sigma)\}\,\d h\\ \notag
&=\int_\R\card[(\R\times \{h\})\cap\graph(F_\sigma)]\,\d h\\ \notag
&=\int_\R i_{]s,t]}(h)\,\d h
\end{align}
where
\begin{align*}
i_{A}:h\in\R\mapsto i_{A}(h)&=\card\{x\in A:\,F_\sigma(x)=h\}\\
&=\card\{x\in A:\,(x,h)\in \graph(F_\sigma)\}
\end{align*}
is the so-called Banach indicatrix, after \cite{Banach}. Notice that in \cite{BY} the result is stated for $s=0$ and $t\in [0,\infty[$. Our statement on a general interval $]s,t]$ is a trivial generalization. Another difference is that we only apply the formula for the occupation measure of $F_\sigma$ in $B$ when $B=\R$. It is hardly more than a simple exercise to rewrite \eqref{eq:bla} for generalized solutions of $F_\sigma=h$, which permits at the same time to forget about the saltus part. For this purpose we introduce the generalized Banach indicatrix. Let $s\leq t$ be in $\overline{\R}$.
\begin{align*}
i_{A}^*:h\in\R\mapsto i_A^*(h)&=\card\{x\in A:\,h\in F_\sigma^*(x)\}\\
&=\card\{x\in A:\,(x,h)\in \graph(F_\sigma^*)\}\in\N\cup\{\infty\}.
\end{align*}
Therefore \eqref{eq:bla} yields
\begin{align}\label{eq:*}
TV_s^t(F_\sigma)=\int i_{]s,t]}^*(h)\,\mathrm{d} h.
\end{align}
A consequence of Theorem 1 in \cite{BY} specifies that not all intersections with the graph need to be considered. Recall first that $\graph^+(F_\sigma^*)$ and $\graph^-(F_\sigma^*)$ have been defined in Section \ref{sec:defexc}. Bertoin and Yor proved that almost surely for $h$, the (generalized) Banach indicatrix $i^*_{]s,t]}$ equals $i^{*,\pm}_{]s,t]}:=i^{*,+}_{]s,t]}+i^{*,-}_{]s,t]}$ where
\begin{align*}
i_{]s,t]}^{*,+}&:h\in\R\mapsto i^+(h)=\card\{x\in \R:\,(x,h)\in \graph^+(F_\sigma^*)\}\in\N\cup\{\infty\},\\
i_{]s,t]}^{*,-}&:h\in\R\mapsto i^-(h)=\card\{x\in \R:\,(x,h)\in \graph^-(F_\sigma^*)\}\in\N\cup\{\infty\}
\end{align*}
(with obvious notation Bertoin and Yor in fact proved the equality $i=i^\pm$ for the non yet generalized indicatrix $i=i^++i^-$.)
This also comes from Theorem 1 in \cite{BY} (where $n^x(t)dx=\lambda^x(t)dx$):
\begin{align}\label{eq:*bis}
TV_s^t(F_\sigma)=\int i_{]s,t]}^{*,\pm}(h)\,\d h\leq \int i_{]s,t]}^{*}(h)\,\d h=TV_s^t(F_\sigma).
\end{align}
With the next result we go further in the analysis.
\begin{pro}\label{pro:altern}
For almost every $h\in \R$, the set $(\R\times \{h\})\cap \graph^*{F_\sigma}$ has cardinal a finite and even integer. In fact for almost every $h\in \R$ it holds $i^*(h)=i^{*,+}(h)+i^{*,-}(h)$ with $i^{*,+}(h)=i^{*,-}(h)$. 
Moreover if $h\neq 0$ the generalized solutions of $F_\sigma=h$ are alternatively crossing positively and negatively, starting with the most-left solution $(x,h)\in \graph^{*,+}(F_\sigma)$ if $h>0$ and $(x,h)\in \graph^{*,-}(F_\sigma)$ if $h<0$.
\end{pro}
\begin{proof}
Due to \eqref{eq:*} the generalized Banach indicatrix is finite for almost every $h \in \R$ and with \eqref{eq:*bis} for almost every $h$ we have $(\R\times \{h\})\cap \graph^*{F_\sigma}=[(\R\times \{h\})\cap \graph^{*,+}{F_\sigma}] \cup [(\R\times \{h\})\cap \graph^{*,-}{F_\sigma}]$. With the generalized intermediate value theorem (Lemma \ref{lem:tvi}) and as $\lim_\infty F_\sigma=\lim_{-\infty} F_\sigma$ we conclude that $i^*(y)$ is an even number for almost every $y$. 
  

More precisely, due to the generalized intermediate value theorem applied to $F_\sigma$ and reminding that this function has limit zero in $\-\infty$ and $+\infty$ the points $x_1<x'_1<x_2<\cdots <x'_{n-1}<x_n<x'_n$  of $\{x\in \R:\, (x,y)\in \graph^{*}(F_\sigma)\}$ are ordered with $x_1,\ldots,x_n\in\graph^{*,+}(F_\sigma)$ and $x'_1,\ldots,x'_n\in\graph^{*,-}(F_\sigma)$ if $h>0$ , $x_1,\ldots,x_n\in\graph^{*,-}(F_\sigma)$ and $x'_1,\ldots,x'_n\in\graph^{*,+}(F_\sigma)$ if $h<0$.
\end{proof}

%

Coming back to \eqref{eq:*}, another direct generalization of Bertoin and Yor's study is
\begin{align}
F_\sigma(t)-F_\sigma(s)=&\int_\R i_{]s,t]}^{*,+}-i_{]s,t]}^{*,-}\,\d h. \label{eq:**}
\end{align}
It will be useful for recovering $F_\mu$ and $F_\nu$ from $F_\sigma$.

For any measurable set $G\subset \R^2$, let $\zeta_G$ denote the following positive measure
\[\zeta_G:E\in \mathcal{T}(\R^2)\mapsto \zeta_G(E)=\int_{-\infty}^{+\infty}\card(x\in \R:\,(x,h)\in E\cap G)\,\d h\]
We consider in particular $G=\graph^*(F_\sigma)$ and $G_+=\graph^{*,+}(F_\sigma)$ and $G_-=\graph^{*,-}(F_\sigma)$ and call $\zeta$, $\zeta_+$ and $\zeta_-$ the corresponding measures. As a consequence of Proposition \ref{pro:altern} and of the concerned definitions we can already state $\zeta=\zeta_+ + \zeta_-$ and $\mathrm{proj}^2_\#\zeta_+=\mathrm{proj}^2_\#\zeta_-=i_{\R}^*(h)/2\,\d h$. The next result indicates the other projections.

 \begin{pro}
The measures $\zeta_+$ and $\zeta_-$ defined as in the previous paragraph
satisfy $\mathrm{proj}^1_\#\zeta_+=\mu$ and $\mathrm{proj}^1_\#\zeta_-=\nu$.
 \end{pro}
 \begin{proof} From \eqref{eq:*} and \eqref{eq:**}, computing the half sum and the half difference we already have
\begin{align}
(TV_+)_s^t(F_\sigma)=\int_\R i^{*,+}_{]s,t]}(h)\,\d h\quad \text{and}\quad(TV_-)_s^t(F_\sigma)=\int_\R i^{*,-}_{]s,t]}(h)\,\d h.
\end{align}
Therefore,
\begin{align}
(TV_+)_s^t(F_\sigma)&=\int_\R\card(x\in ]s,t]:\, (x,h)\in \graph^+(F^*_\sigma))\,\d h\\
&=\int_{-\infty}^{+\infty}\card(x\in \R:\,(x,h)\in (]s,t]\times \R)\cap \graph^+(F^*_\sigma))\,\d h\\
&=\mathrm{proj}^1_\#\zeta_+(]s,t])
\end{align}
and $(TV_-)_s^t(F_\sigma)=\mathrm{proj}^1_\#\zeta_-(]s,t])$. Recall moreover the definitions

\begin{align}
\left\{
\begin{aligned}
TV(F)_s^t&=\sup_{m\in \N,\, s=r_0<\ldots<r_{m+1}=t}\sum_{j=0}^m|F(r_{k+1}-F(r_k)|\\
(TV_+)_s^t(F)&=\sup_{m\in \N,\,s=r_0<\ldots<r_{m+1}=t}\sum_{j=0}^m[F(r_{k+1})-F(r_k)]^+\\
(TV_-)_s^t(F)&=\sup_{m\in \N,\,s=r_0<\ldots<r_{m+1}=t}\sum_{j=0}^m[F(r_{k+1})-F(r_k)]^-
\end{aligned}
\right.
\end{align}
where $[u]^+=\frac12(|u|+u)$ and $[u]^-=\frac12(|u|-u)$ are the positive and negative parts of $u$. As $\mu \perp \nu$, there exists $P$ with $\mu(P)=\mu(\R)$ and $\nu(P)=0$. By outer regularity for every $\eps$ there exists an open set $U\supset P$ such that  $|\mu(P)-\mu(\R)|=0$ and $\nu(U)<\eps/2$. 

Finally from the $\sigma$-additivity of measures and the fact $\mu(]a,b-1/n])\to \mu(]a,b[)$ for every $a<b$ there exists a finite union $J=\bigcup_i]a_i,b_i]\subset U$ of semi open intervals such that $\nu(J)<\eps$ and $|\mu(J)-\mu(\R)|<\eps$. The partition of $\R$ given by $J$ permits us to check that the positive total variation of $F_\sigma$ is greater than $\sum [\sigma(]a_i,b_i])]^+$ which is greater than $\sum_i \sigma(]a_i,b_i])=\mu(J)-\nu(J)\geq \mu(\R)-2\eps$. This holds for every $\eps\geq 0$ so that $TV_+(F_\sigma)\geq TV(F_\mu)=\mu(\R)$. Symmetrically $TV_+(F_\sigma)\geq TV(F_\nu)=\nu(\R)$. Finally as $TV_+(F_\sigma)+TV_-(F_\sigma)=TV(F_\sigma)=TV(F_\mu-F_\nu)\leq TV(F_{\mu})+TV(F_{\nu})= (\mu+\nu)(\R)=|\sigma|(\R)$ we conclude with $TV_+(F_\sigma)= TV(F_\mu)$ and $TV_-(F_\sigma)= TV(F_\nu)$. The same identity is correct on $]s,t]$ with the same proof (for instance $(TV_+)_s^t(F_\sigma)= (TV)_s^t(F_\mu)=\mu(]s,t])=\mathrm{proj}^1_\#\zeta_+(]s,t]$).
\end{proof}


\begin{proof}[Proof of Theorem \ref{thm:construction}]
We proved that $i_\R$ is almost surely an even integer and $\int_{-\infty}^{+\infty}i_\R(h)\,\d h=TV(F_\sigma)=TV(F_\mu)+TV(F_\nu)=2$. Due to Proposition \ref{pro:altern} $(X,H)$ is concentrated on $\graph^+(F^*_\sigma)$ and is distributed as $\zeta_+$. Similarly $(Y,H)$ is concentrated on $\graph^-(F^*_\sigma)$ and distributed as $\zeta_-$. Finally the law of $X$ is $\mu$ and the law of $Y$ is $\nu$.
\end{proof}

\subsection{Proof that a transport plan concentrated on a monotone set is the excursion coupling}
In this subsection we call monotone a set $S$ with arches that do not cross, do not connect and have same orientation when they are nested. In what follows we prove $\emph{3}\Rightarrow \emph{4}$ of the Main Theorem, i.e. that measures concentrated on a monotone set are the excursion coupling of their marginals (this is correct even though these measures do not have finite first moment). We prove this first in the case $\mu\perp \nu$ and prove the general case on page \pageref{par:general}.

\paragraph{Proof of $\emph{3}\Rightarrow \emph{4}$ for measures $\mu\perp\nu$}

The generalized $\graph^*{(F_\sigma)}$, $\pi$ and the related objects are still defined as above and we still assume $\mu\perp \nu$. We define $\Gamma$ (that depends on $\sigma$) as
\begin{align}\label{eq:gamma}
\Gamma=\bigcup_{h>0, h\in C}\bigcup^{i^*_\R(h)/2}_{k=1}\{(x^h_{2k-1},x^h_{2k})\}\cup\bigcup_{h<0, h\in C}\bigcup^{i^*_\R(h)/2}_{k=1}\{(x^h_{2k},x^h_{2k-1})\}
\end{align}
where $C$ is the set of levels $h$ such that $i^*_\R(h)=i^{*,\pm}_\R(h)$ and $x^h_1<x^h_2<\cdots <x^h_{i^*_\R-1}<x^h_{i^*_\R}$  are the points of $\{x\in \R:\, (x,h)\in \graph^{*}(F_\sigma)\}$. We stated in Proposition \ref{pro:altern} that $C$ has full measure with respect to $i_\R^*(\d h)$. Therefore, with respect to the definition of the excursion coupling, we have $\pi(\Gamma)=1$. 

\begin{rem}\label{rem:3-4}
We could prove that $\Gamma$ is monotone (the arches of $\Gamma$ do not cross, do not connect and have the same orientation when they are nested). Since $\pi(\Gamma)=1$, this would correspond to the implication $\emph{4}\Rightarrow \emph{3}$. This is correct and can be proved directly but our proof of the Main Theorem goes $\emph{4}\Rightarrow \emph{1}\Rightarrow \emph{2}\Rightarrow \emph{3}$.
\end{rem}

\begin{figure}\label{fig:arches}
\begin{center}
   \def\svgwidth{12cm}
   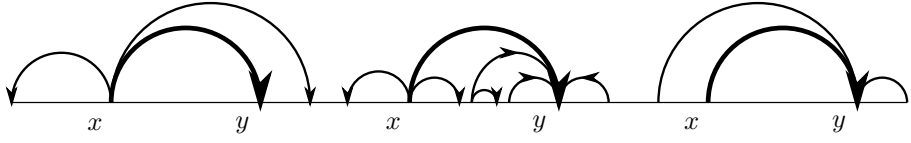
\caption{Transport routes inside and outside $]x,y[$ and at points $x$ and $y$}\label{fig:arches}
\end{center}
\end{figure}


\begin{pro}
Let $\mu$ and $\nu$ be mutually singular measures of $\p_1(\R)$ and $\gamma$ be a monotone transport plan in $\ma(\mu,\nu)$. Let $S$ be a monotone set with $\gamma(S)=1$. Then $\Gamma\cap S$ is still monotone and satisfies $\gamma(\Gamma\cap S)=1$.
\end{pro}
\begin{proof}
It is not a priori known that $\gamma(\Gamma)=1$ and this statement is in fact clearly equivalent to the proposition result. Let $S$ be as in the statement. We will define $S\supset S_1\supset S_2\supset S_3$ such that $S_k\setminus S_{k-1}$ has measure zero for $\gamma$ and $S_3\subset \Gamma$. Hence we will have $\gamma(\Gamma)\geq \gamma(\Gamma \cap S)\geq\gamma(S_3 \cap S)=\gamma(S)=1$.

Let $\gamma$ and $S$ be as in the statement. Let $\Delta$ be the diagonal $\{(x,y)\in \R^2:\,x=y\}$ and $A=\pr^1(\Gamma)$ and $B=\pr^2(\Gamma)$. Let $S_1$ be  $S \cap (\R^2\setminus \Delta)$ and $S_2=S_1\cap (A\times \R)\cap (\R\times B)$. Finally $S_3=S_2\setminus \{(x,y)\in \R^2:\,\gamma(x,y)=0,\mu(x)>0,\nu(y)>0\}$.

\emph{Proof of $\gamma(S_3)=\gamma(S)$}. We already recalled in Lemma \ref{lem:coincide} that transport plans $\gamma\in \ma(\mu,\nu)$, whose arches do not connect takes the form $\gamma=(\id\times \id)_\#(\mu\wedge \nu)+\gamma_0$ where $\gamma_0(\Delta)=0$. As $\mu\perp \nu$ we have $\gamma(\Delta)=0$ and $\gamma(S_1)=\gamma(S)$. We have $\gamma(A\times \R)=\mu(A)=\pi(\Gamma)=1$. Similarly, $\gamma(\R\times B)=1$ so that $\gamma(S_2)=\gamma(S_1)=1$. Finally there are countably many pairs $(x,y)$ with $\mu(x)>0$ and $\nu(y)>0$. Thus $\gamma(\{(x,y)\in \R^2:\,\gamma(x,y)=0,\mu(x)>0,\nu(y)>0\})=0$ and $\gamma(S_3)=\gamma(S_2)$=1.

\emph{Proof of $S_3\subset \Gamma$.} Let $(x,y)$ be in $S_3$. Without loss of generality we can assume $x<y$ ($x=y$ became impossible as $S$ was replaced by $S_1$). Let us first prove
\begin{align}\label{eq:bla1}
\gamma((\R\setminus [x,y])\times ]x,y[)=0\qquad\text{and}\qquad \gamma(]x,y[\times (\R\setminus [x,y]))=0.
\end{align}
 If not there exists $(x',y')\in \spt(\gamma)$ with $x'\notin [x,y]$ and $y'\in ]x,y[$ (or the same property inverting the role of $x'$ and $y'$). As $S$ is dense in $\spt(\gamma)$ the same is correct for a point $(x'',y'')\in S$, which leads to a contradiction with the monotonicity of $S$, whose arches should not cross. 
 
 \emph{Case 1:} Assume $\mu(x)=\nu(y)=0$ ; the complementary case is considered further in case 2. Then $F_\sigma(y)-F_\sigma(x)=\mu([x,y])-\nu([x,y])=\gamma([x,y]\times \R)-\gamma(\R\times [x,y])=\gamma([x,y]^2)-\gamma([x,y]^2)=0$. Moreover, a similar argument as for \eqref{eq:bla1} shows that for any  $y'\in ]x,y]$ we have 
 \begin{align}\label{eq:bla2}
\nu(]x,y'])=\gamma([x,y']\times ]x,y'])\leq \mu(]x,y']).
 \end{align}
In fact $\gamma((\R\setminus [x,y'])\times ]x,y'])\neq 0$ would imply that there exists some $(x'',y'')\in S$ with $y''\in ]x,y']\subset ]x,y[$ and $x''\notin [x,y']$. If $x''<x$ this contradicts the fact that arches $(x,y)$ and $(x'',y'')$ do not cross. If $x''>y'$ the latter fact or the one that nested arches have the same orientation is violated. From \eqref{eq:bla2} we find $F_\sigma(y')\geq F_\sigma(x)=F_\sigma(y)$ for every $y'\in ]x,y[$. As $\mu(x)=0$ and $\nu(x)=0$ ($\nu(x)>0$ is not possible because $x$ is in $A$ so that it can not be an abscise of the decreasing part of $\graph^-(F_\sigma)$) the multivalued function $F_\sigma^{*}$ is single valued at $x$. Hence, from the definition of $A$ involved in $S_2$ it follows that $h=F_\sigma(x)$ must be an element of $C$. Therefore, since the level $h$ cuts $\graph(F_\sigma^*)$ in points of $\graph^+(F_\sigma^*)$ or $\graph^-(F_\sigma^*)$ it is not possible to have $F_\sigma(x)=F_\sigma(y')=F_\sigma(y)$ for $y'\in ]x,y[$ (In a neighborhood of $y'$ we would have $F_\sigma\geq h$). It follows that $x$ and $y$ are consecutive zeros of $F^*_\sigma=h$. Thus $(x,y)\in \Gamma$.
 
\emph{Case 2:} We want to finalize the inclusion $S_3\subset \Gamma$ looking at the pairs $(x,y)\in S_3\subset S$ where $x$ or $y$ is an atom of $\sigma$. Since the arches of $S$ do not cross $(x,y)$ at least one of the two is true: \emph{i)} $]-\infty,x[\times \{y\}$ has empty intersection with $S$), or \emph{ii)} $\{x\}\times]y,\infty[$ has empty intersection with $S$ (recall that here $x<y$), i.e the arches of the left and right patterns of Figure \ref{fig:arches} can not all be in $S$. Without loss of generality we will assume that $]-\infty,x[\times \{y\}$ is empty. This corresponds to the two first patterns from on Figure \ref{fig:arches}. Adapting the argument of case 1 we find $F_\sigma(y)-F_\sigma(x)=\mu(]x,y])-\nu(]x,y])\leq 0$ and $F_\sigma(y)-F_\sigma(x^-)=\mu([x,y])-\nu([x,y])\geq 0$ in place of $F_\sigma(y)-F_\sigma(x)=0$. Thus $F_\sigma(y)$ is in $[F_\sigma(x^-),F_\sigma(x)]$. Let $x'$ be in $]x,y]$. We have $F_\sigma(y)-F_\sigma(x')=\mu(]x',y])-\nu(]x',y])\leq 0$. This is due to the fact that the arches starting from $]x',y]$ have the same orientation as $(x,y)$ and do not cross it. Moreover, for $\nu$ some additional mass in $y$ could arrive from $]y,+\infty[$. We want to prove that $F_\sigma(y)-F_\sigma(x')< 0$ is also true for any $x'\in ]x,y[$. If $y$ is not an atom we can proceed as before starting with $F_\sigma(y)\in C$ and the arche $(x,y)$ with $F_\sigma^*(x)\ni F_\sigma(y)$. Therefore we assume that $y$ is an atom of $\nu$. If $F_\sigma(y)\in ]F_\sigma(x^-),F_\sigma(x)[$ we will be able to conclude easily that there exists $h\in C\cap ]F_\sigma(x^-),F_\sigma(x)[\cap ]F_\sigma(y^-),F_\sigma(y)[$ and we conclude as we did twice before (on $[x,y]$ the generalized function $F_\sigma^*$ can not only touch the level $h$ but it must cut it, which is not possible because $(x,y)\in \Gamma$ for the level $h$). In the other case $F_\sigma(x^-)<F_\sigma(x)=F_\sigma(y)$ and all the mass arriving in $y$ comes from $]x,y]$. Therefore $(x,y)$ may be an element of $S_2$ but not of $S_3$, a contradiction. This case can not happen and we proved $(x,y)\in \Gamma$ in all the other cases. Finally we proved $S_3\subset \Gamma$.

%


\end{proof}

\begin{pro}
Let $\mu$ and $\nu$ be singular measures and $\pi\in\ma(\mu,\nu)$ the excursion coupling and $\Gamma$ as defined in \eqref{eq:gamma}. Let $\pi'\in \ma(\mu,\nu)$ be another coupling concentrated on $\Gamma$. Then $\pi'=\pi$.
\end{pro}
\begin{proof} 
Let $A$ be the set of atomic points of $\mu$ and $B$ the set of atomic points of $\nu$. From the definition of $\Gamma$ we see that the set $\Gamma\setminus (A\times \R)$ is contained in the graph of a function from $\R\setminus A$ to $\R$. The same is true, inverting the coordinates for $\Gamma\setminus (\R \times B)$. Hence 
the measures $\pi$ and $\pi'$ coincide on the set $\R^2\setminus (A\times B)=[(\R\setminus A)\times \R] \cup [\R\times (\R\setminus B)]$ that we denote by $E$, i.e
\[\pi(\cdot\cap E)=\pi'(\cdot\cap E).\]
Therefore we aim at proving that $\pi$ and $\pi'$ coincide on the countable set $E^c=A\times B$. We will in fact prove $\pi\{(a,b)\}=\pi'\{(a,b)\}$ for every $(a,b)\in \Gamma$. Let $(a,b)$ be in $\Gamma$. Let us assume without loss of generality $a<b$. We further assume $F_\sigma(a^-)\geq F_\sigma(b)$ so that according to the definition of $\Gamma$ the transport plans $\pi$ and $\pi'$ are concentrated on it and the mass of $\mu$ on $[a,b[$ must be transported onto $]a,b]$. This writes $\mu([a,b[)=\pi([a,b[\times ]a,b])$ and we have also $\mu([a,b[)=\pi'([a,b[\times ]a,b])$. Reasoning similarly we obtain $\nu(]a,b[)=\pi([a,b[\times ]a,b[)$. Therefore
\begin{align}\label{eq:ahah}
\pi([a,b[\times\{b\})=\mu([a,b[)-\nu(]a,b[).
\end{align}
The construction of $\Gamma$ associates the route $(a,b)\in \Gamma$ with $a<b$ to some level $h>0$. The generalized intermediate value theorem, Lemma \ref{lem:tvi} permits us to derive $F_\sigma(a')>h$ for every $a'\in ]a,b[$ so that it also holds $F_\sigma(a'^-)\geq F_\sigma(b)$. Therefore \eqref{eq:ahah} can be written for any $a'$ such that $(a',b)\in \Gamma$ in place of $a$. Recall that it also holds for $\pi'$ in place of $\pi$. We will be done if we can prove $\pi(]a,b[\times \{b\})=\pi'(]a,b[\times \{b\})$. This is in fact correct because
\begin{align*}
\pi(]a,b[\times \{b\})&=\pi((]a,b[\times \{b\})\cap \Gamma)\\
&=\sup\{\pi(([a',b[\times \{b\})\cap \Gamma)):\,a'\in \R, (a',b)\in\Gamma,\,a<a'<b\}\\
&=\sup\{\pi'(([a',b[\times \{b\})\cap \Gamma)):\,a'\in \R, (a',b)\in\Gamma,\,a<a'<b\}\\
&=\pi'((]a,b[\times \{b\})\cap \Gamma)=\pi'(]a,b[\times \{b\}).
\end{align*}
\end{proof}

\paragraph{Proof of $\emph{3}\Rightarrow \emph{4}$ for general measures $\mu$ in $\nu$} \label{par:general}
We no longer assume $\mu\perp \nu$. In this case we know from Lemma \pageref{lem:coincide} that any $\gamma$ satisfying $\emph{3}$ in the Main Theorem can be written in the form $(\id\times\id)_\#(\mu\wedge \nu)+\gamma_0$ where $\gamma_0\in \ma(\mu-(\mu\wedge \nu), \nu-(\mu\wedge \nu))$ satisfies $\gamma_0(\Delta)=0$. We are in the situation of the last paragraph because $\mu-(\mu\wedge \nu)\perp \nu-(\mu\wedge \nu)$ and $\gamma_0$ satisfies $\emph{3}$: since we know that there is a monotone set $S$ with $\gamma(\R^2\setminus S)=0$ it also holds $\gamma_0(\R^2\setminus S)=0$. From the discussion above we obtain that $\gamma_0$ is the excursion coupling of $\mu-(\mu\wedge \nu)$ and $\nu-(\mu\wedge \nu)$. This exactly implies that $\gamma$ is the excursion coupling of $\mu$ and $\nu$.

\section{Final elements of proof of the Main Theorem and its corollary}\label{sec:proof}

\begin{proof}[Proof of the Main Theorem]
The structure of the proof is the following:
\begin{itemize}
\item The set of measures satisfying $\emph{1}$ (the solutions to the $L^{1^-}$ problem) is not empty.
\item The set of measures satisfying $\emph{2}$ (the solutions to the $L^{1,q}$ problem) is not empty.
\item Assumption $\emph{1}$ implies $\emph{3}$ and assumption $\emph{2}$ implies $\emph{3}$ (see Section \ref{sec:1to3}).
\item Assumption $\emph{3}$ implies $\emph{4}$ (see Section \ref{sec:3to4}).
\item There is a unique and well-defined coupling satisfying $\emph{4}$ (see Theorem \ref{thm:construction}).
\end{itemize}
Therefore, if $\pi$ satisfies $\emph{4}$ it equals any coupling satisfying $\emph{1}$, respectively $\emph{2}$. As these sets are not empty, if $\pi$ satisfies $\emph{4}$ it also satisfies $\emph{1}$, respectively $\emph{2}$.

We proved everything except the two first existence statements. They will be obtained as consequences of Lemma \ref{lem:continuous} that is proved in this section. 
\end{proof}

Il order to prove that there exists a solution to the $L^{1^-}$ limit transport problem (property $\emph{1}$) it suffices to remind of two elementary facts. First, any sequence $(\pi_n)_{n\in \mathbb{N}}$ in $\ma(\mu,\nu)$ admits cluster points. The set $\ma(\mu,\nu)$ is indeed a compact set for the weak topology, as a simple consequence of Prokhorov Theorem and of the fact that it is closed. Second, as we recalled in the introduction, the set $\ma_p^*(\mu,\nu)$ is not empty for every $p\in ]0,1[$. These two elements permit us to conclude that there exists at least one element $\pi\in \ma(\mu,\nu)$ that satisfies $\emph{1}$. 
However, the non emptiness of $\ma_p^*(\mu,\nu)$ relied on the following Lemma \ref{lem:continuous} through the standard argument of optimization. We prove it now for the sake of completeness and in order to prepare the proof of Lemma \ref{lem:continuous2}.

\begin{lem}\label{lem:continuous}
Let $p,\,\mu$ and $\nu$ be as in the statement. Let $p$ be in $]0,1[$ and $\mu,\,\nu$ be two mesures with finite moment of order $p$. The function 
\[T_p:\pi\in \ma(\mu,\nu)\to \iint|y-x|^p\,\d\pi(x,y)\]
is continuous on $\ma(\mu,\nu)$, endowed with the weak topology of laws on $\R^2$. 
\end{lem}
\begin{proof}
For every $n\in\N^*$ we decompose $(x,y)\to|y-x|^p$ as the sum of a bounded continuous function and a reminder function:
\[|y-x|^p=\underbrace{(|y-x|^p\wedge n)}_{:=f_{n,p}(x,y)}+\underbrace{(|y-x|^p-n)^+}_{:=h_{n,p}(x,y)}\]
where $(u)^+=\frac12(u+|u|)$ denotes the positive part of $u$. Since $n\geq 1$ we have
\[h_n(x,y))\leq (|x|^p-n/2)^+ +(|y|^p-n/2)^+\leq |x|^p\mathds{1}_{|x|\geq n/2}+|y|^p\mathds{1}_{|y|\geq n/2} \]
so that $\iint  h_n(x,y) \d \pi(x,y)\leq \int_{|x|^p\geq n/2}|x|^p\d\mu(x)+ \int_{|y|^p\geq n/2}|y|^p\d\nu(y)$ for every $\pi\in \ma(\mu,\nu)$. Thus $s_{n,p}:=\sup\{\iint  h_n \d \pi:\,\pi\in \ma(\mu,\nu)\}\to_{n\to \infty} 0$. Let us finish proving that $T_p$ is continuous at $\pi$. Let $\eps$ be a positive real number and $n$ large enough to make $s_{n,p}$ smaller that $\eps/4$. Now
\begin{align*}
\left|F_p(\pi)-F_p(\pi')\right| &\leq \left| \iint f_{n,p}\, \d (\pi-\pi')\right|+ \left|\iint h_{n,p}\, \d (\pi-\pi')\right|\\
&\leq \left|\iint f_{n,p}\, \d (\pi-\pi')\right|+2s_{n,p}.
\end{align*}
As $f_n$ is continuous and bounded this estimate proves that there exists a neighborhood $U$ of $\pi$ such that for every $\pi'\in U$ it holds $|T_p(\pi')-T(\pi)|\leq \eps$.
\end{proof}

\begin{rem}
Until now, even though it is clear from the context and a consequence of the Main Theorem (implication $\emph{2}\Rightarrow \emph{1}$) we never proved that a so-called solution to the $L^{1^-}$ limit transport problem is a solution to the $L^1$ transport problem. The next result may be invoked to prove it directly. 

\begin{lem}\label{lem:continuous2}
The function $(p,\pi)\in]0,1]\times\ma(\mu,\nu)\to F_p(\pi)$ is continuous.
\end{lem}
\begin{proof}
Let us fix $(p,\pi)$. For another pair $(p',\pi')$ we have
\begin{align*}
|F_p(\pi)-F_{p'}(\pi')|&\leq |F_p(\pi)-F_{p'}(\pi)|+|F_{p'}(\pi)-F_{p'}(\pi')|\\
&\leq 2s_{n,1}+|F_p(\pi)-F_{p'}(\pi)|+\left|\iint [(1+|y-x|)\wedge n]\, \d (\pi-\pi')\right|.
\end{align*}
Let $\eps>0$. Due to the dominated convergence theorem the term $|F_p(\pi)-F_{p'}(\pi)|$ is smaller than $\eps/4$ for $p'\in ]0,1]$ in a neighborhood of $p$. For $n$ large enough $s_{n,p}$ is smaller that $\eps/4$. For this $n$ fixed we see that the last term is smaller than $\eps/4$ when $\pi'$ is in a certain neighborhood of $\pi$.
\end{proof}
\end{rem}

Let us now prove that the set of solutions to the $L^{1,q}$ secondary transport problem is not empty. With what precedes including Lemma \ref{lem:continuous} we indeed know that $\ma_1^*(\mu,\nu)$ is not empty and that it is closed. Let $(\pi_n)_{n\in \N}$ be a minimizing sequence for $\pi\mapsto \iint |y-x|^q\d \pi(x,y)$ on $\ma_1^*(\mu,\nu)$. This is a continuous function so that $\ma^{**}_{1,q}$ is not empty. 

We finish the section with the proof that seemingly stronger results are equivalent to $\emph{1}$ and $\emph{2}$ in the Main Theorem.

\begin{proof}[Proof of Corollary \ref{cor:main_cor}]

The Main Theorem $\emph{4}\Rightarrow \emph{1}$ applies to any subsequence $(p_{\varphi(n)})_{n\in \N}$ that is increasing. Therefore, in the compact set $\ma(\mu,\nu)$ any subsequence of $(\pi_n)_{n\in \N}$ possesses an increasing subsequence converging to $\pi$. This proves $\emph{1'}$.

Let $q'$ be an exponent smaller than $1$. Then the Main Theorem $\emph{4}\Rightarrow \emph{2}$ applies for $q'=q$.
\end{proof}

\section{Concluding bibliographic remarks and perspectives}\label{sec:remarks}

\begin{rem}[the $L^0$ transport problem]\label{rem:zero}
For $p=0$ the cost function is $d^0(x,y)=\mathds{1}_{\{x\neq y\}}$ where $d$ is the distance on $\R$. This corresponds to the coupling problem that defines the total variation of $\mu$ and $\nu$. A coupling $\pi$ is a solution if and only if it writes $(\id\times \id)_\#(\mu\wedge \nu)+\pi_0$ and $\min_{\pi\in \ma(\mu,\nu)} T_0(\pi)=\max|\mu(A)-\nu(A)|=1/2|\mu-\nu|(\R)$.
\end{rem}

\begin{rem}[the $L^p$ transport problem for $p<0$]\label{rem:coulomb}
For $p<0$ the cost function $d^p$ where $d$ is the distance on $\R$ is singular on the diagonal $\Delta=\{(x,y)\in\R^2:\,x=y\}$ where it takes the value $+\infty$. Notice that it writes $\ell(|y-x|)$ where $\ell:[0,\infty)\to \R^+\cup\{\infty\}$ is decreasing, convex and has limits $\infty$ and $0$. This type of cost, including the Coulomb cost $c:(x,y)\mapsto |y-x|^{-1}$ has been thoroughly studied by Cotar, Friesecke and Kl\"uppelberg in \cite{CFK} with the purpose of determining the joint distribution of electronic particles on their orbitals. Their Theorem 3.1 states an existence and uniqueness result for measures admitting a density. In \S4.1 they conduct a precise study of the one dimensional case in the spirit of \cite{GMcC, McC}, the same geometric spirit that is also inspiring us in the present paper. Concerning the Coulomb type costs, notice that the assumption $\mu=\nu$ is the natural one for the chemical application. In their Theorem 4.8 the authors completely characterize the optimal transport for an absolutely continuous measures $\mu=\nu$ with positive density. As for $p>0$ this solution does not depend on the particular value of $p<0$ (or of $l$). 

However, there is no uniqueness of the optimal transport plan in general, as can be seen for instance with the example $\mu=\nu=\frac13(\delta_{-1}+\delta_{0}+\delta_{1})$. Moreover the example $\mu=(1/2)(\delta_0+\delta_{1+\Phi})$, $\nu=(1/2)(\delta_{-\Phi}+\delta_{1})$ with $\Phi=(1+\sqrt{5})/2$ is similar to Example \ref{ex:main}: For $p<-1$ the problem admits a unique solution $\pi^0$, for $p>-1$ another transport plan $\pi^1$ is the unique solution and for $p=1$ the solutions are the plans $(\pi^\lambda)_{\lambda\in [0,1]}$ defined by $\pi^\lambda=\lambda\pi^1+(1-\lambda)\pi^0$.
\end{rem}

\begin{rem}[On the solution of the Monge problem selected in \cite{DiML}]\label{rem:DiML}
It is clear that if there exists $a\in \R$ such that $\mu$ and $\nu$ are concentrated in $[a,+\infty[$ and $]-\infty,a]$ respectively, then $\ma^*_1(\mu,\nu)$ coincide with the set of all transport plans between $\mu$ and $\nu$. In fact for measures with finite first moment $\ma^*_1(\mu,\nu)=\ma(\mu,\nu)$ if and only if there exist such a real $a\in \R$ splitting the supports of $\mu$ and $\nu$ (but the symmetric situation $\spt(\mu)\in [a,+\infty[$ is possible). For general measures $\mu,\nu\in \p_1(\R)$ the set $\ma^*_1(\mu,\nu)$ has recently be described by Di Marino and Louet in \cite{DiML}. This recent paper concerns another way to select a special element of $\ma^*_1(\mu,\nu)$ when the entropy parameter in the entropic regularized Monge problem tends to zero. Note that the resulting coupling is different from ours. If the measure $\mu,\,\nu$ are as in the beginning of this remark, the plan of Di Marino and Louet is $\mu\times \nu$. We obtain the decreasing rearrangement, i.e the law of $(F_\mu,1-F_\nu)$ seen as a random vector on $([0,1],\lambda)$.
\end{rem}

\begin{rem}[On the Skorkhod problem for unbiased Brownian motions]
One motivation to our paper was to better understand a work by Last, M\"orters and Thorisson \cite{LMT} and reformulate their construction in the framework of the optimal transport theory. In this paper eternal Brownian motions $(B_t)_{t\in \R}$ starting in $\mu$ are embedded onto $\nu$ with non-negative random time $T$ such that $(B_{t-T})_{t \in \R}$ is an eternal Brownian motion independent of $T$. The authors define a coupling similar to our excursion coupling but for two $\sigma$-finite random measures on $\R$. This solution minimizes $\mathbb{E}(\varphi(T))$ for any concave function $\varphi:\R^+\to\R^+$. Comparing with our case that concerns deterministic measure $\mu,\,\nu\in \p_1(\R)$ one can conjecture that among couplings $(X,Y)$ with law in $\ma(\mu,\nu)$ that satisfy the constraint $X\leq Y$, the excursion coupling is the one minimizing $\mathbb{E}(Y-X)^p$ for every $p\in]0,1[$.
\end{rem}

\begin{rem}\label{rem:BY}
Bertoin and Yor \cite{BY} established their deterministic formulae in relation with an important chapter of Stochastic Calculus. The occupation measure of (the continuous part of) a real semimartingale turns out to be a random absolutely countinuous measure with density expressed in terms of local times, that are quantities described by the Meyer--Tanaka formula. The work of Bertoin and Yor provides analogue results in the deterministic word of functions with finite variation.
\end{rem}

\begin{rem}[Sharpness of the assumptions] The assumptions in the Main Theorem are by no mean claimed to be sharp. For property \emph{\ref{un}}, a rough analysis of the proof seems to indicate that the family of costs $(c_\eps)_{\eps\in ]0,1[}$ defined by $c_\eps(x,y)=|y-x|^{1-\eps}$ can be replaced by any family $(c'_\eps)_{\eps\in ]0,1[}$ of type $c'_\eps(x,y)=\ell_\eps(|y-x|)$ where it is assumed $\ell_\eps(d)\to_{\eps\to 0^+} d$ for every $d\geq 0$ and $\ell_\eps$ is increasing and strictly concave. Concerning property \emph{\ref{deux}}, any $c'$ of the same type as before should play the same role as $|y-x|^q$. Finally, the fact that $\mu$ and $\nu$ have a finite first moment should not be necessary to state the equivalence between \emph{\ref{trois}} and \emph{\ref{quatre}} (see also Remark \ref{rem:3-4}). 
\end{rem}

\begin{rem}[$L^{1-}$ limit transport problem in Euclidean spaces]\label{rem:geod}

The transport problem for the Monge distance cost in Euclidean and further in some more general geodesic spaces is a research stream with a rich history. It recently culminated with the optimal transport proof by Cavalletti and Mondino of the L\'evy-Gromov inequality \cite{CM}. It is intimately connected to the Monge problem on the real line because under appropriate assumptions on the space and the marginal measures, any optimal transport plan can be disintegrated as a mixture of one dimensional transport plans concentrated on disjoint geodesic rays.

A natural question with respect to the present paper is the existence and uniqueness of a solution to the $L^{1^-}$ limit problem in Euclidean spaces. One can moreover conjecture that such a cluster solution can be disintegrated in a way that the transport along the geodesic rays always is the excursion coupling. A similar result has been proved in \cite{AP} for the $L^{1^+}$ limit problem and the quantile coupling.
\end{rem}

\begin{rem}[Generating $\pi$ by picking a random point on a tree]\label{rem:tree}
A popular construction in probability is to associate a random tree with a random function (or process) $f$. Typically in the construction of the continuum random Brownian tree \cite{Ald3} a random tree is associated to a random excursion. However this topological construction is purely deterministic. Let $f$ be function defined on an interval $I$. We write $x\sim x'$ if and only if $f(x)=f(x')$ and $y\in ]x,x'[\Rightarrow f(y)\geq f(x)$. The tree is the quotient space $I/\sim$. The fact that we conducted a similar operation with our function $F_\mu-F_\nu$ yields an appealing interpretation. In place of choosing a random point $h$ on the $y$-axis with density $i^*$ and continue selecting uniformly an excursion among $i^*(h)\in \N$ possible we could directly choose randomly a point on the associated tree according to the length measure, i.e the Hausdorff measure of dimension $1$. Doing this we come closer to the classical simulation of the quantile coupling where a point $h$ is chosen on the tree $[0,1]$ (a segment) according to the length measure (the Lebesgue measure). While it is clear that the measure on the tree is the correct one if $\mu$ and $\nu$ are simple, e.g finite sums of atoms, or such that $F=F_\mu-F_\nu$ is of class $\mathcal{C}^1$ with finitely many changes of monotonicity,  it is not is the general case. We leave it as a conjecture that the quotient measure on the tree corresponding to our generalized function $F^*=(F_\mu-F_\nu)^*$ always is the length measure.
\end{rem}



\bibliographystyle{alphaabbr}
\bibliography{basebib_mountain}

\def\cprime{$'$}
\begin{thebibliography}{{Ban}25}

\bibitem[Ald93]{Ald3}
D.~Aldous.
\newblock The continuum random tree. {III}.
\newblock {\em Ann. Probab.}, 21(1):248--289, 1993.

\bibitem[AP03]{AP}
L.~Ambrosio and A.~Pratelli.
\newblock Existence and stability results in the {$L^1$} theory of optimal
  transportation.
\newblock In {\em Optimal transportation and applications ({M}artina {F}ranca,
  2001)}, volume 1813 of {\em Lecture Notes in Math.}, pages 123--160.
  Springer, Berlin, 2003.

\bibitem[{Ban}25]{Banach}
S.~{Banach}.
\newblock {Sur les lignes rectifiables et les surfaces dont l'aire est finie.}
\newblock {\em {Fundam. Math.}}, 7:225--236, 1925.

\bibitem[BY14]{BY}
J.~Bertoin and M.~Yor.
\newblock Local times for functions with finite variation: two versions of
  {S}tieltjes change-of-variables formula.
\newblock {\em Bull. Lond. Math. Soc.}, 46(3):553--560, 2014.

\bibitem[CFK13]{CFK}
C.~Cotar, G.~Friesecke, and C.~Kl\"{u}ppelberg.
\newblock Density functional theory and optimal transportation with {C}oulomb
  cost.
\newblock {\em Communications on Pure and Applied Mathematics}, 66(4):548--599,
  2013.

\bibitem[CM17]{CM}
F.~Cavalletti and A.~Mondino.
\newblock Sharp and rigid isoperimetric inequalities in metric-measure spaces
  with lower {R}icci curvature bounds.
\newblock {\em Invent. Math.}, 208(3):803--849, 2017.

\bibitem[DML18]{DiML}
S.~Di~Marino and J.~Louet.
\newblock The entropic regularization of the {M}onge problem on the real line.
\newblock {\em SIAM J. Math. Anal.}, 50(4):3451--3477, 2018.

\bibitem[GM96]{GMcC}
W.~Gangbo and R.~J. McCann.
\newblock The geometry of optimal transportation.
\newblock {\em Acta Math.}, 177(2):113--161, 1996.

\bibitem[LMT14]{LMT}
G.~Last, P.~M\"{o}rters, and H.~Thorisson.
\newblock Unbiased shifts of {B}rownian motion.
\newblock {\em The Annals of Probability}, 42(2):431--463, 2014.

\bibitem[McC01]{McC}
R.~J. McCann.
\newblock Polar factorization of maps on {R}iemannian manifolds.
\newblock {\em Geom. Funct. Anal.}, 11(3):589--608, 2001.

\bibitem[Mon81]{Monge}
G.~Monge.
\newblock M\'emoire sur la th\'eorie des d\'eblais et des remblais.
\newblock {\em Histoire de l'acad\'emie {R}oyale des {S}ciences de {P}aris},
  1781.

\bibitem[Vil03]{Vi1}
C.~Villani.
\newblock {\em Topics in optimal transportation}, volume~58 of {\em Graduate
  Studies in Mathematics}.
\newblock American Mathematical Society, Providence, RI, 2003.

\end{thebibliography}
\end{document}